\documentclass[11pt, twoside, a4paper]{article}

\usepackage[utf8]{inputenc}
\usepackage[usenames,dvipsnames]{xcolor}

\usepackage{amsmath}   				
\usepackage{amssymb}		
\usepackage{amsthm}     

\usepackage{mathtools}
				
\usepackage{graphicx}
\usepackage{authblk}
\usepackage{caption}

\usepackage{lpic}

\usepackage{marvosym} 

\usepackage[normalem]{ulem}

\usepackage[hidelinks]{hyperref}

\newcommand{\R}{\mathbb{R}}

\newcommand\restr[2]{{
		\left.\kern-\nulldelimiterspace 
		#1 
		\vphantom{\big|} 
		\right|_{#2} 
}}

\newcommand{\ud}{\mathrm{d}}

\newtheorem{theorem}{Theorem}[section]
\newtheorem{corollary}[theorem]{Corollary}
\newtheorem{lemma}[theorem]{Lemma}
\newtheorem{proposition}[theorem]{Proposition}

\theoremstyle{definition}
\newtheorem{definition}[theorem]{Definition}
\theoremstyle{definition}
\newtheorem{remark}[theorem]{Remark}
\theoremstyle{definition}

\numberwithin{equation}{section}
\numberwithin{table}{section}
\numberwithin{figure}{section}

\title{\bf Traveling waves in reaction-diffusion-convection equations with combustion nonlinearity}
\date{}

\author[1]{Pavel Dr\'{a}bek}
\author[2]{Michaela Zahradn\'{i}kov\'{a}\footnote{Corresponding author. Orcid ID: 0000-0003-3303-6969. Email: zahram05@prf.jcu.cz}}
\affil[1]{\footnotesize Department of Mathematics and NTIS, Faculty of Applied Sciences, University of West Bohemia, Univerzitn\'{i}~8, 301 00 Plze\v{n}, Czech Republic}
\affil[2]{\footnotesize Department of Mathematics, Faculty of Science, University of South Bohemia,\newline Brani\v{s}ovsk\'{a} 1760, 370 05 \v{C}esk\'{e} Bud\v{e}jovice, Czech Republic}

\makeatletter

\renewcommand{\@fnsymbol}[1]{%
	\ifcase#1 \or {\,\Letter\!} \or\textasteriskcentered\or \textasteriskcentered\textasteriskcentered 
	\else\@ctrerr\fi}
\makeatother

\begin{document}

\maketitle

\begin{minipage}{0.85\textwidth}\parindent=15.5pt

{\small
\noindent {\bf Abstract.}
This paper concerns the existence and properties of traveling wave solutions to reaction-diffusion-convection equations on the real line. We consider a general diffusion term involving the $p$-Laplacian and combustion-type reaction term.
We extend and generalize the results established for $p=2$ to the case of singular and degenerate diffusion.
Our approach allows for non-Lipschitz reaction as well.
We also discuss the shape of the traveling wave profile near equilibria, assuming power-type behavior of the reaction and diffusion terms.

\smallskip
	
\smallskip
	
}

\end{minipage}

\section{Introduction}
\label{s:Intro}

In this paper, we study the existence of traveling wave solutions for the reaction-diffusion-convection equation
\begin{equation}
\label{eqv_partial}	
	v_t=\left[D(v)|v_x|^{p-2}v_x\right]_x + (H(v))_x + g(v), \quad x \in \R, \,t\geq 0.
\end{equation}
Here $p>1$, $D \in C^1(0,1)$ is a diffusion coefficient with $D>0$ in $(0,1)$, $H \in C^1[0,1]$ represents a nonlinear convective flux function, and the reaction term $g\in C[0,1]$ 
satisfies
\begin{equation} \label{g_comb}
	g(v) =0 \,\mbox{ in }\, [0,\theta], \quad g(v)>0 \,\mbox{ in }\, (\theta,1), \quad g(1)=0
\end{equation}
for some $\theta \in (0,1)$.
Let us emphasize that the diffusion coefficient $D$ might degenerate or have a singularity at one or both endpoints. 
A traveling wave solution (t.w.s. for short) is a solution of the form $v(x,t)=u(x-ct)$, $c \in \R$, connecting the equilibria 0 and 1. Besides the unknown profile $u$, the constant speed $c$ also needs to be determined. 

Reaction-diffusion equations with and without convective effects are used to describe a variety of phenomena in biology, chemistry and physics. We refer to \cite{GK} for an overview of classical applications in the case $p=2$, involving different types of reaction terms.
Nonlinearity $g$ satisfying \eqref{g_comb} arises in combustion models with an ignition temperature assumption, in which the reaction starts only after the temperature reaches the threshold value $\theta$, cf.~\cite{Ber}.
The diffusion coefficient $D$ is usually assumed to be strictly positive in $[0,1]$. However, modeling of certain dispersal phenomena suggests also coefficients with degenerations or singularities at equilibria, see, e.g., \cite{KingMcCabe}, \cite{SG-Maini}.   
More recently, $p$-Laplacian-type diffusion has also been considered in literature, see \cite{Audrito,Enguica,Hamydy} and their references. The $p$-Laplacian operator itself appears for instance in models derived from the power-type Darcy's law, cf. \cite{BGK}.

If $H$ is constant, i.e., no convective effects are present, equation \eqref{eqv_partial} possesses a unique 
nonincreasing t.w.s. with a positive wave speed $c^*$.
For $p=2$, this result was proved in \cite{MMM} assuming $g \in C[0,1]$ and $D \in C^1[0,1]$ strictly positive in $[0,1]$. Furthermore, if stronger regularity conditions on $g$ are imposed, 
the t.w.s. is strictly monotone on $\R$. 
In our previous work \cite{DrZa-tw}, we proved the existence of a nonincreasing solution to \eqref{eqv_partial} without convection in an even more general setting. In particular, we considered a discontinuous coefficient $D$ with finitely many jumps in $(0,1)$ and with possible singularities and/or degenerations at equilibria 0 and 1. This required a new concept of non-classical non-smooth solution, but the existence result remained the same also for $p>1$.

In \cite{MM}, the authors investigate the effect of convection on the existence of t.w.s. in the case $p=2$ and $D>0$ in $[0,1]$.
They derive existence and non-existence results based on whether, in some sense, $H$ prevails over or is weak compared to the terms $D$ and $g$. 
Our paper focuses on the same phenomena but we consider arbitrary $p>1$ as well as a more general diffusion coefficient $D$ (possibly degenerate or singular at 0 and 1) and \mbox{a non-Lipschitz reaction $g$}.
The main contribution of our paper consists in showing how the results from \cite{MM} extend to our more general case.
We also note that our method of proof differs from that in \cite{MM}.


The paper is organized as follows. In Section \ref{s:mainresults}, we present the definition of possibly non-smooth traveling wave profile $u$ and our main existence and nonexistence results. 
In Section \ref{s:Reduction}, we first discuss monotonicity property of the wave profile, which then allows us to reduce the second-order problem for the unknown profile $u$ to a first-order one on a bounded interval.
Proofs of the main results are provided in Sections \ref{s:Nonex} (nonexistence) and \ref{s:Ex} (existence).
Auxiliary lemmas referenced in Sections \ref{s:Reduction} and \ref{s:Ex} can be found in the Appendices \ref{ap:A} and \ref{ap:B}, respectively.
Finally, Section \ref{s:Asymptotics} is dedicated to discussions concerning the asymptotic properties of solutions, while assuming power-type behavior of the terms $D$ and $g$.

\section{Main results}
\label{s:mainresults}

Without loss of generality, we assume $H(0)=0$ and write 
\begin{equation*}
	H(u) \coloneqq \int_{0}^{u} h(s)\,\ud s.
\end{equation*}
Here $h(u)=\frac{\ud}{\ud u} H(u)$, $u \in [0,1]$, is the convective velocity.
When looking for traveling wave solutions $v(x,t)=u(\xi)$, $\xi \coloneq x-ct$, the partial differential equation \eqref{eqv_partial} becomes a boundary value problem for the second-order ordinary differential equation on the real line 
\begin{equation}
	\label{2BVP-u}
	\begin{cases}
		\left(D(u)|u'|^{p-2}u'\right)' +\left(c+h(u)\right)u' +g(u)=0,\\
		u(-\infty)=1,\,\, u(+\infty)=0
	\end{cases}
\end{equation}
where $u' \coloneqq \frac{\ud u}{\ud \xi}$ stands for the derivative with respect to the wave coordinate $\xi$. 
Since the equation in \eqref{2BVP-u} is autonomous, its solutions are invariant under translations. Therefore, we can always normalize solutions of \eqref{2BVP-u} as $u(0)=\theta$.

Before presenting the main results of this paper, we first provide the definition of solution to \eqref{2BVP-u} on $\R$. This concept accounts for situations when the profile $u$ reaches one or both equilibria 0 and 1.

\begin{definition}
\label{def:sol}
A continuous function $u: \R \to [0,1]$ is a solution of \eqref{2BVP-u} if
\begin{itemize}
	\item [(a)]
	$u \in C^1(I_u)$, $I_u \coloneq \{\xi \in \R: 0<u(\xi)<1\}$, and the equation in \eqref{2BVP-u} holds at every point of $I_u$;
	\item [(b)]
	the function $\xi \mapsto D(u(\xi))|u'(\xi)|^{p-2}u'(\xi)$ is continuous on $\R$ and $D(u(\xi))|u'(\xi)|^{p-2}u'(\xi) \to 0$ as $u(\xi) \to 0$ and $u(\xi) \to 1$;
	\item [(c)]
	(boundary conditions)
	$u(\xi) \to 1$ as $\xi \to -\infty$ and $u(\xi) \to 0$ as $\xi \to +\infty$.
\end{itemize}
\end{definition}

\begin{remark}
At the beginning of Section \ref{s:Reduction}, we prove that $I_u$ is an open interval, bounded or unbounded, and that $u'(\xi)<0$ for all $\xi \in I_u$.
If $p=2$, $D \in C^1[0,1]$ with $D>0$ in $[0,1]$ and $g$ is a Lipschitz function on $[0,1]$, then $I_u=\R$ and $u \in C^2(\R)$ is a classical solution, cf. \cite{MM}. 
Note that in this case we have $u'(\xi) \to 0$ as $\xi \to \pm \infty$.

On the other hand, if $p \ne 2$ and $D \in C^1(0,1)$ is singular or degenerate at 0 and/or 1, then $I_u$ might be an interval of finite length, i.e., $I_u=(\xi_1,\xi_2)$, $\xi_1, \xi_2 \in \R$. Moreover, the derivative $u'(\xi_i)$, $i=1,2$, need not exist. We provide detailed discussion of these cases in Section \ref{s:Asymptotics} based on asymptotic properties of $D$ and $g$ which yield different shapes of traveling wave profile.
\end{remark}

\begin{remark} \label{r:diff-v}
Notice that it follows from Definition \ref{def:sol}\,(a) that the function
\begin{equation*}
	\xi \mapsto D(u(\xi))|u'(\xi)|^{p-2}u'(\xi)
\end{equation*}
belongs to $C^1(I_u)$. In particular, the function $\xi \mapsto |u'(\xi)|^{p-2}u'(\xi)$ also belongs to $C^1(I_u)$ due to our assumption $D \in C^1(0,1)$, $D>0$ in $(0,1)$.
\end{remark}

\medskip

In what follows, we denote
\begin{equation*}
	h_m \coloneqq \min_{u \in [0,1]} h(u)
\end{equation*}
and $p'$ stands for the exponent conjugate to $p$, i.e., $p'=\frac{p}{p-1}$.
Furthermore, we assume 
\begin{equation}
\label{int-Dg-finite}
	\int_0^1 D^{p'-1}(u)g(u)\,\ud u < +\infty.
\end{equation}

\begin{theorem}[Nonexistence]
\label{t:nonex}
Let
\begin{equation}
	\label{nonex-condition}
	H(\theta) \geq \theta h_m+\left(p' \int_{0}^1 D^{p'-1}(u) g(u) \,\ud u\right)^\frac{1}{p'}.
\end{equation}
Then the boundary value problem \eqref{2BVP-u} has no solution for any $c>-h_m$.
If strict inequality holds in \eqref{nonex-condition}, there is no solution for any $c \geq -h_m$.
\end{theorem}

In Section \ref{s:Nonex}, we prove that $c\geq -h(0)$ is a necessary condition for the existence of solutions. An immediate consequence is the following corollary, which addresses the nonexistence of solution for any real value of $c$.

\begin{corollary}
\label{c:nonex}
If strict inequality holds in \eqref{nonex-condition} and $h_m=h(0)$, then \eqref{2BVP-u} has no solution for any $c \in \R$.
\end{corollary}

Let
\begin{equation}
\label{k-def}
k=k(p)=
\begin{cases}
	\frac{1}{2^{p'-1}-1} & \text{if } 1<p<2, \\
	1 & \text{if } p=2, \\
	\frac{p'}{p'-1+\frac{1+p'(p'-1)^\frac{1}{p'-2}+(p'-1)^\frac{p'}{p'-2}}{\left(1+(p'-1)^\frac{1}{p'-2}\right)^{p'}}} & \text{if } p>2.
\end{cases}
\end{equation}
Then $k=k(p)$ is a continuous function in $(1,+\infty)$ satisfying 
\begin{equation*}
	\lim\limits_{p\to 1+} k(p)=0 \quad \mbox{ and } \quad \lim\limits_{p\to +\infty} k(p)=\frac{1}{2}.
\end{equation*}

\begin{theorem}[Existence]
\label{t:existence}
Let
\begin{equation}
\label{ex-condition}
	H(1) \leq h_m + \left(k(p) \int_{0}^1 D^{p'-1}(u) g(u)\,\ud u\right)^\frac{1}{p'}.
\end{equation}
Then there exists a unique $c=c^*>-h_m$ such that the boundary value problem \eqref{2BVP-u} has a unique (up to translation) solution $u=u(\xi)$. Moreover, the solution $u$ is strictly decreasing on $I_u$ and $c^*$ satisfies
\begin{equation}
\label{estimate-c}
	c^* < \frac{1}{\theta} \left[\left(p'\int_{0}^1 D^{p'-1}(u) g(u)\,\ud u\right)^\frac{1}{p'}-H(\theta)\right]-h_m.
\end{equation}
\end{theorem}

\medskip

Clearly, if $h_m \leq 0$ it follows from the above theorem that the unique wave speed $c^*$ is positive. The following result concerns the existence of a positive wave speed in the case $h_m>0$.

\begin{theorem}[Positive wave speed $c$]
\label{t:existence-c>0}
If $h(u)> 0$, $u \in [0,1]$, and
\begin{equation}
	\label{ex-condition-c>0}
	H(1) \leq \left(k(p) \int_{0}^1 D^{p'-1}(u) g(u)\,\ud u\right)^\frac{1}{p'},
\end{equation}
then $c^*>0>-h_m$.	
\end{theorem}

\begin{remark}
The expression for $k(p)$ given in \eqref{k-def} is an optimal value of the constant in Theorems \ref{t:existence} and \ref{t:existence-c>0}. For $p=2$ it coincides with the estimates derived in \cite{MM} by a different approach.
\end{remark}

\section{Reduction to a first order problem}
\label{s:Reduction}

In this section, we establish our main tool for investigating the existence and nonexistence of solutions to \eqref{2BVP-u}. In particular, we show that \eqref{2BVP-u} can be transformed into a first-order boundary value problem.
First, we prove that each solution to \eqref{2BVP-u} is a decreasing function in $I_u$.

\begin{proposition}
\label{p:prop-u}
Let $u$ be a solution of \eqref{2BVP-u}. There exist $-\infty\leq \xi_1<\xi_2\leq +\infty$ such that
$u \equiv 1$ in $(-\infty,\xi_1]$, $u \equiv 0$ in $[\xi_2,+\infty)$ and
$u'(\xi)<0$ for any $\xi \in (\xi_1,\xi_2)$.
\end{proposition}

\begin{proof}
First, we show that the derivative of a solution to \eqref{2BVP-u} does not vanish in the set $I_u=\{\xi \in \R: 0<u(\xi)<1\}$. Indeed, let $\xi_0 \in I_u$ be such that $0<u(\xi_0) \leq \theta$. If $u'(\xi_0)=0$ then it follows from Lemma \ref{l:A1} that the boundary conditions in \eqref{2BVP-u} are not satisfied, a contradiction. Now consider $\xi_0 \in I_u$, $\theta < u(\xi_0)<1$ with $u'(\xi_0)=0$.
Then 
\[
\left.\left(D(u(\xi)|u'(\xi)|^{p-2}u'(\xi))\right)'\right|_{\xi=\xi_0}=-g(u(\xi_0))<0.
\]
It follows from Lemma \ref{l:A2} that $\xi_0$ must be the point of strict local maximum of $u$ and therefore $\lim\limits_{\xi \to -\infty} u(\xi) \ne 1$, again a contradiction. 

Next we prove that $u'(\xi)<0$ for all $\xi \in I_u$, i.e., the solution cannot ``switch'' from 0 to 1 and back again finitely many times (while still satisfying the boundary conditions).
To this end, we observe that $c>-H(1)$ is a necessary condition for the existence of solution to \eqref{2BVP-u}. Indeed, integrating the equation in \eqref{2BVP-u} we obtain
\begin{equation*}
\begin{split}
D(u(\xi)) |u'(\xi)|^{p-2} u'(\xi) &{}- D(u(\hat{\xi}))|u'(\hat{\xi})|^{p-2}u'(\hat{\xi})
+
c(u(\xi)-u(\hat{\xi})) \\
&+H(u(\xi))-H(u(\hat{\xi})) 
+ \int_{\hat{\xi}}^{\xi} g(u(\zeta))\,\ud \zeta = 0, \quad \xi,\hat{\xi} \in \R.	
\end{split}
\end{equation*}
Passing to the limits $\xi \to +\infty$, $\hat{\xi} \to -\infty$ and taking into account parts (b) and (c) of Definition \ref{def:sol} yields
\begin{equation*}
c+H(1)-H(0)=\int_{-\infty}^{+\infty} g(u(\zeta))\,\ud\zeta.
\end{equation*}
Since $H(0)=0$ and the integral on the left-hand side is positive, we conclude that $c>-H(1)$.

Suppose that there exist $\underline{\xi}, \bar{\xi} \in \R$ such that $u(\underline{\xi})=0$, $u(\bar{\xi})=1$ and $u'(\xi)>0$ for all $\xi \in (\underline{\xi},\bar{\xi})$.
Integrating the equation in \eqref{2BVP-u} from $\underline{\xi}$ to $\bar{\xi}$ and employing the same arguments as above, we arrive at
\begin{equation*}
c+H(1) 
= - \int_{\underline{\xi}}^{\bar{\xi}} g(u(\zeta))\,\ud \zeta < 0,	
\end{equation*}
i.e., $c<-H(1)$, a contradiction.

Therefore, 
there exist $-\infty\leq \xi_1<\xi_2\leq +\infty$ such that
$u \equiv 1$ in $(-\infty,\xi_1]$, $u \equiv 0$ in $[\xi_2,+\infty)$ and
$u'(\xi)<0$ for any $\xi \in (\xi_1,\xi_2)$.
This concludes the proof.
\end{proof}

In particular, it follows from Proposition \ref{p:prop-u} that $I_u=(\xi_1,\xi_2)$ is an open interval, bounded or unbounded.

Following substitutions from \cite[p.\,174]{Enguica}, we set
\begin{equation}
	\label{w(u)}
	-w(u) \coloneqq D(u)|u'|^{p-2}u'.
\end{equation}
Since $u'(\xi)<0$ for all $\xi \in (\xi_1,\xi_2)$, we have $w=w(u)>0$ in $(0,1)$ and $w$ satisfies
\begin{equation*}
	\frac{1}{p' D^{p'-1}(u)} \frac{\ud }{\ud u} w^{p'}(u) - (c+h(u)) \left(\frac{w(u)}{D(u)}\right)^{p'-1}+g(u)=0, \quad u \in (0,1).
\end{equation*}
Put
\begin{equation*}
	y(u) \coloneqq w^{p'}(u) >0.
\end{equation*}
Then $y=y(u)$ solves the equation
\begin{equation}
	\label{ODE-y}
	y'(u)=p'\left[(c+h(u))(y^+(u))^\frac{1}{p}-f(u)\right], \quad u \in (0,1),
\end{equation}
where $y'=\frac{\ud y}{\ud u}$, $y^+(u)\coloneqq \max\{y(u),0\}$ and
\begin{equation}
\label{def-f}
f(u)\coloneqq D^{p'-1}(u)g(u).
\end{equation}
In terms of $y$, part (b) of Definition \ref{def:sol} translates to
\begin{equation}
	\label{y-01}
	y(0) \coloneq \lim\limits_{u\to 0+} y(u)=0, \quad y(1) \coloneq \lim\limits_{u\to 1-} y(u)=0.
\end{equation}
It follows from \eqref{w(u)} that
\begin{equation*}
	\frac{\partial \xi}{\partial u} = -\left(\frac{D(u)}{w(u)}\right)^{p'-1} 
\end{equation*}
and therefore, since $u(0)=\theta$,
\begin{equation}
	\label{xi}
	\xi(u)= -\int_{\theta}^{u} \left(\frac{D(s)}{w(s)}\right)^{p'-1}\,\ud s 
	= -\int_{\theta}^{u} \frac{D^{p'-1}(s)}{y^{\frac{1}{p}}(s)}\,\ud s, \quad u \in (0,1).
\end{equation}
Since $\xi=\xi(u)$ maps $(0,1)$ onto $(\xi_1,\xi_2)$, we have
\begin{equation}
	\label{xi1-xi2}
	\xi_1=-\int_{\theta}^{1} \frac{D^{p'-1}(s)}{y^\frac{1}{p}(s)}\,\ud s
	\quad \mbox{ and } \quad 
	\xi_2=\int_{0}^{\theta} \frac{D^{p'-1}(s)}{y^\frac{1}{p}(s)}\,\ud s.
\end{equation}
It follows from the above calculations that the existence of a monotone solution to \eqref{2BVP-u} implies the existence of a positive solution to \eqref{ODE-y}, \eqref{y-01} which, in addition, satisfies \eqref{xi1-xi2}.
Proceeding as in \cite[Proposition 3.3]{DZ22} where $h \equiv 0$, it can be shown that these problems are equivalent.
We thus have the following assertion.

\begin{proposition} \label{p:equiv}
Let \eqref{int-Dg-finite} hold. A function $u:\R \to [0,1]$ is a unique solution (up to translation) of \eqref{2BVP-u} in the sense of Definition \ref{def:sol} if and only if $y:[0,1] \to \R$ is a unique positive solution of \eqref{ODE-y}, \eqref{y-01}.
\end{proposition}

\section{Proof of nonexistence results}
\label{s:Nonex}

Due to Proposition \ref{p:equiv},
to prove the nonexistence results, it suffices to show that the first-order boundary value problem \eqref{ODE-y}, \eqref{y-01} does not admit positive solutions for the given values of $c$.
First, we notice that
\begin{equation}
	\label{c-NC}
	c\geq -h(0)
\end{equation}
is a necessary condition for the existence of a positive solution of \eqref{ODE-y}, \eqref{y-01}. Indeed, if $c<-h(0)$ then, by the continuity of $h$, there exists $\delta>0$ such that $c<-h(u)$ for all $u\in[0,\delta]$. Integrating the equation \eqref{ODE-y} over $[0,\delta]$ and taking into account $y_c(0)=0$ together with $c+h(u)<0$ in $[0,\delta]$, we arrive at
\begin{equation*}
	y_c(\delta)=p'\int_0^\delta (c+h(u))(y_c^+(u))^\frac{1}{p}\,\ud u <0,
\end{equation*}
a contradiction with the positivity of solution $y_c=y_c(u)$.

\medskip
\noindent
\textbf{Proof of Theorem \ref{t:nonex}.}
We proceed by contradiction and assume that $c>-h_m$ and $y_c=y_c(u)$ is a positive solution of \eqref{ODE-y}, \eqref{y-01}.
Integrating the equation \eqref{ODE-y} over $(\theta,1)$ and using \eqref{y-01} yields
\begin{equation}
	\label{y-theta}
	y_c(\theta)=-p'\int_{\theta}^1 (c+h(\sigma))\left(y_c(\sigma)\right)^\frac{1}{p}\,\ud \sigma 
	+
	p'\int_{\theta}^1 f(\sigma)\,\ud \sigma 
	<
	p' \int_{0}^1 f(\sigma)\,\ud \sigma,
\end{equation}
where $f$ is given by \eqref{def-f}.
On the other hand, since $f \equiv 0$ on $(0,\theta)$ the equation \eqref{ODE-y} is separable on $(0,\theta)$. Using \eqref{y-01} we obtain
\begin{equation}
	\label{y-theta2}
	y_c^{\frac{1}{p'}}(\theta)=c\theta+H(\theta).
\end{equation}
It follows from \eqref{y-theta}, \eqref{y-theta2} and the condition \eqref{nonex-condition} that
\begin{equation*}
	\left(p' \int_0^1 f(\sigma)\,\ud \sigma\right)^\frac{1}{p'}
	> y_c^{\frac{1}{p'}}(\theta)=c\theta+H(\theta)>-h_m\theta + H(\theta)
	\geq \left(p' \int_0^1 f(\sigma)\,\ud \sigma\right)^\frac{1}{p'},
\end{equation*}
a contradiction.

Assuming strict inequality in \eqref{nonex-condition} and $c \geq -h_m$, we would arrive at
\begin{equation*}
	\left(p' \int_0^1 f(\sigma)\,\ud \sigma\right)^\frac{1}{p'}
	\geq y_c^{\frac{1}{p'}}(\theta)=c\theta+H(\theta) \geq -h_m\theta + H(\theta)
	> \left(p' \int_0^1 f(\sigma)\,\ud \sigma\right)^\frac{1}{p'},
\end{equation*}
again a contradiction. 
This concludes the proof.
\qed

\section{Proof of existence results}
\label{s:Ex}

\textbf{Proof of Theorem \ref{t:existence}.}
We first prove the statement of Theorem \ref{t:existence} assuming that $h_m=0$, i.e., we will show that if 
\begin{equation*}
	H(1) \leq \left(k(p) \int_{0}^1 D^{p'-1}(u) g(u)\,\ud u\right)^\frac{1}{p'},
\end{equation*}
then there exists a unique positive value $c=c^*$ for which \eqref{2BVP-u} admits a solution.
This result can then be applied to the case of a more general $h \in C[0,1]$ with $h_m \ne 0$ by means of a suitable shift, 
as discussed at the end of this section.

Thanks to the equivalence established in Proposition \ref{p:equiv}, we proceed by investigating the initial value problem
\begin{equation}
\label{bivp-comb}
\begin{cases}
	y'_c(u)=p'\left[(c+h(u))(y_c^+(u))^\frac{1}{p}-f(u)\right], \quad u \in (0,1), \\
	y_c(1)=0,
\end{cases}	
\end{equation}
where $f$ is given by \eqref{def-f}.
Let $c \geq 0$. Since $c+h(u) \geq 0$ for all $u \in [0,1]$, the function
\begin{equation*}
	y \mapsto (c+h(u))(y^+)^\frac{1}{p}, \quad u \in [0,1],
\end{equation*}
satisfies one-sided Lipschitz condition and it follows from \cite[Lemma 4.1 and Lemma 4.3]{DZ22}, where we replace $c$ by $c+h(u)$, that \eqref{bivp-comb} has a unique global solution $y_c=y_c(u)$ defined on $[0,1]$. Our aim is to show that there exists $c>0$ such that $y_c(u)>0$ if $u \in (0,1)$ and $y_c(0)=0$.

First, let us observe that $f(u)>0$ in $(\theta,1)$ implies that
\begin{equation}
\label{y_c>0-near1}
	y_c(u)>0 \,\, \mbox{ for }\,\, u \in (\theta,1),
\end{equation}
and
\begin{equation}
\label{y_c(theta)<p'int}
	y_c(\theta)=-p' \int_{\theta}^1 (c+h(\sigma))(y_c(\sigma))^\frac{1}{p} \,\ud \sigma +p'\int_{\theta}^1 f(\sigma)\,\ud \sigma
	< p'\int_0^1 f(\sigma)\,\ud\sigma.
\end{equation}

According to Lemma \ref{l:y<H}, for any $p>1$ we have
\begin{equation}
\label{yH-theta}
	y_0^\frac{1}{p'}(\theta) > H(\theta). 
\end{equation}
In particular, $y_0(\theta)>0$ and hence there exists $0<\delta\leq\theta$ such that $y_0(u)>0$ for $u \in (\theta-\delta,\theta)$. 
Since $f \equiv 0$ on $(0,\theta)$, $y_0=y_0(u)$ solves the equation
\begin{equation*}
	y_0'(u)=p'h(u)(y_0(u))^\frac{1}{p}, \quad u \in (\theta-\delta,\theta).
\end{equation*}
Separating variables, we obtain for $u\in(\theta-\delta,\theta)$
\begin{equation*}
	y_0^\frac{1}{p'}(\theta)-y_0^\frac{1}{p'}(u)=H(\theta)-H(u),
\end{equation*}
i.e.,
\begin{equation*}
	y_0^\frac{1}{p'}(u)-H(u)=y_0^\frac{1}{p'}(\theta)-H(\theta)>0
\end{equation*}
by \eqref{yH-theta}. It follows that $\delta=\theta$ and
\begin{equation*}
	y_0^\frac{1}{p'}(u)>0 \quad \mbox{ for all}\,\, u \in [0,\theta].
\end{equation*}
Therefore, 
\begin{equation}
\label{y0>0}
	y_0(u)>0 \quad \mbox{ for all }\,\, u \in [0,1).
\end{equation}

Set
\begin{equation*}
	c^*\coloneqq \sup \{c>0: y_c(u)>0 \mbox{ for all } u\in(0,1)\}.
\end{equation*}
It follows from \eqref{y_c>0-near1}, \eqref{y0>0} and the continuous dependence of the solution to \eqref{bivp-comb} on the parameter $c$ that the set $\{c>0: y_c(u)>0 \mbox{ for all } u\in(0,1)\}$ is nonempty and $c^*>0$.
If $c^*=+\infty$ then there exist $c_n \to +\infty$ and corresponding $y_{c_n}=y_{c_n}(u)>0$, $u \in (0,1)$, which satisfy
\begin{equation*}
	y'_{c_n}(u)=p'(c_n+h(u))(y_{c_n}(u))^\frac{1}{p}, \quad u \in(0,\theta).
\end{equation*}
Separating variables yields
\begin{equation}
	\label{y_cn_sep}
	\left(y_{c_n}(u)\right)^\frac{1}{p'}=(y_{c_n}(\theta))^\frac{1}{p'}+c_n(u-\theta)+H(u)-H(\theta), \quad u \in (0,\theta),
\end{equation}
and from \eqref{y_c(theta)<p'int} we get
\begin{equation*}
	y_{c_n}(\theta)<p'\int_0^1 f(\sigma)\,\ud\sigma<+\infty.
\end{equation*}
Therefore, the right-hand side in \eqref{y_cn_sep} tends to $-\infty$, a contradiction.
Hence $0<c^*<+\infty$.

Next we prove that $y_{c^*}(u)>0$, $u\in(0,1)$, $y_{c^*}(0)=0$. Indeed, by the continuous dependence of \eqref{bivp-comb} on the parameter $c$ and the definition of $c^*$, the solution $y_{c^*}=y_{c^*}(u)$ must vanish somewhere in the interval $[0,\theta]$. Let $\eta \in [0,\theta]$ be the largest zero of $y_{c^*}$. 
It follows from the comparison argument that solutions of \eqref{bivp-comb} decrease (not strictly) with $c$. This can be easily shown as in \cite[Lemma 4.5 and Corollary 4.6]{DZ22} by replacing $c$ with $c+h(u)$.
If $\eta>0$ then for $c<c^*$ we have $y_c(u) \geq 0$ on $(0,\eta)$ and hence from
\begin{equation*}
	y'_c(u)=p'(c+h(u))(y_c(u))^\frac{1}{p}, \quad u\in (0,\eta),
\end{equation*}
we again deduce
\begin{equation}
	\label{y_c-eta}
	0 \leq  (y_c(u))^\frac{1}{p'}=(y_c(\eta))^\frac{1}{p'}+c(u-\eta)+H(u)-H(\eta).
\end{equation}
Since for $c\to c^*$ we have $y_c(\eta) \to y_{c^*}(\eta)=0$ by continuous dependence on parameter, for any fixed $u\in(0,\eta)$ there exists $c<c^*$, $(c^*-c)$ sufficiently small, such that
\begin{equation*}
	(y_c(\eta))^\frac{1}{p'}+c(u-\eta)+H(u)-H(\eta)<0
\end{equation*}
($h_m=0$ implies that $H$ is nondecreasing), which contradicts \eqref{y_c-eta}. Hence $\eta=0$ and we have $y_{c^*}(u)>0$ in $(0,1)$, $y_{c^*}(0)=0$. By the uniqueness for the initial value problem \eqref{bivp-comb}, this solution is also a unique solution of the boundary value problem \eqref{ODE-y}, \eqref{y-01} with $c=c^*$.

Finally, we show that positive solutions of \eqref{bivp-comb} do not vanish at 0 for values of $c \ne c^*$.
Assume by contradiction that there exists $\hat{c} \ne c^*$ such that $y_{\hat{c}}=y_{\hat{c}}(u)>0$ solves \eqref{bivp-comb} in $(0,1)$, $y_{\hat{c}}(0)=0$. The definition of $c^*$ yields $\hat{c}<c^*$. Separating variables in the equation in \eqref{bivp-comb} on  $(0,\theta)$, we obtain
\begin{equation*}
	y_{\hat{c}}^\frac{1}{p'}(\theta)=\hat{c}\theta+H(\theta)
\end{equation*}
and also 
\begin{equation}
\label{y_c*-theta}
	y_{c^*}^\frac{1}{p'}(\theta)=c^*\theta+H(\theta).
\end{equation}
Hence $y_{\hat{c}}(\theta)<y_{c^*}(\theta)$. On the other hand, the comparison argument applied to \eqref{bivp-comb} yields $y_{\hat{c}}(u)\geq y_{c^*}(u)$, $u\in[0,1]$. 
This follows from \cite[Corollary 4.6]{DZ22} with $c_1=\hat{c}+h(u)$ and $c_2=c^*+h(u)$.
In particular, $y_{\hat{c}}(\theta)\geq y_{c^*}(\theta)$, a contradiction.

It follows from \eqref{y_c*-theta} together with \eqref{y_c(theta)<p'int} that
\begin{equation*}
	c^*=\frac{1}{\theta}\left(y_{c^*}^\frac{1}{p'}(\theta)-H(\theta)\right)<\frac{1}{\theta}\left[\left(p'\int_0^1 D^{p'-1}(u)g(u)\,\ud u\right)^\frac{1}{p'}-H(\theta)\right],
\end{equation*}
i.e., \eqref{estimate-c} holds. This concludes the proof for $h_m=0$.

\medskip

If $h_m \ne 0$, we can consider a new convective velocity $\tilde{h}(u) \coloneqq h(u)-h_m$, $u \in [0,1]$. Then $\tilde{h}_m \coloneqq \min_{u \in [0,1]} \tilde{h}(u)= 0$ and $\tilde{H}(u)=\int_0^u \tilde{h}(s)\,\ud s=H(u)-h_m u$ is a nondecreasing function. Setting $\tilde{c} \coloneqq c+h_m$, the equation in \eqref{2BVP-u} becomes
\begin{equation*}
	\left(D(u)|u'|^{p-2}u'\right)'+(\tilde{c}+\tilde{h}(u))u'+g(u)=0
\end{equation*}
and we can apply the above reasoning to prove the existence of a unique positive value $\tilde{c}^*$ assuming that
\begin{equation*}
	\tilde{H}(1) \leq \left(k(p) \int_{0}^1 D^{p'-1}(u) g(u)\,\ud u\right)^\frac{1}{p'}.
\end{equation*}
Since $\tilde{H}(1)=H(1)-h_m$, we immediately see that condition \eqref{ex-condition} yields a unique value $c^*=\tilde{c}^*-h_m>-h_m$ corresponding to the problem with convective velocity $h$ and the estimate \eqref{estimate-c} holds.
\qed

\medskip
\noindent
\textbf{Proof of Theorem \ref{t:existence-c>0}.}
For $h_m>0$ we can carry out the proof exactly as in the case $h_m=0$, see the proof of Theorem \ref{t:existence}.
In particular, statements concerning the initial value problem \eqref{bivp-comb} remain valid and the positivity of $h$ justifies the use of Lemma~\ref{l:y<H}.
Therefore, if \eqref{ex-condition-c>0} holds we conclude that $c^*>0>-h_m$. 
\qed

\section{Asymptotic analysis of the wave profile}
\label{s:Asymptotics}

In this section, we discuss asymptotic behavior of the solution $u=u(\xi)$ to \eqref{2BVP-u} as $\xi \to \pm \infty$. 
Our aim is to determine whether the solution attains 0 and/or 1 (or neither of them).
To this end, we study the convergence of the integrals from \eqref{xi1-xi2}, and hence the boundedness of the interval $(\xi_1,\xi_2)$.
For technical reasons and for the sake of brevity, we assume power-type behavior of $D$ and $g$ near equilibria 0 and~1.

In what follows, we consider $H(u)>0$, $u \in (0,1]$, and profiles with $c^*>0$. 
For notational brevity, we use the following notation:
for $s_0 \in \R$ we write
\begin{equation*}
	\phi_1(s) \sim \phi_2(s) \, \mbox{ as } s \to s_0
	\quad
	\mbox{if and only if}
	\quad 
	\lim\limits_{s\to s_0} \frac{\phi_1(s)}{\phi_2(s)} \in (0,+\infty).
\end{equation*}

\subsection{Asymptotics near 0}

Let us assume that $D(u) \sim u^\alpha$ as $u \to 0+$ for some $\alpha \in \R$.
Thanks to $f \equiv 0$ in $[0,\theta]$, we have
\begin{equation*}
	y_{c^*}^{\frac{1}{p'}}(u)=c^*u+H(u), \quad u \in (0,\theta),
\end{equation*}
and due to the assumption $H\in C^1[0,1]$, $H(u)>0$ together with $c^*>0$, we have $y_{c^*}^\frac{1}{p'}(u) \sim u$ as $u \to 0+$.
Let us recall that
\begin{equation*}
	\xi_2=\int_{0}^{\theta} \frac{D^{p'-1}(s)}{y_{c^*}^{1/p}(s)}\,\ud s.
\end{equation*}
Since
\begin{equation}
	\label{as0}
	\int_0^u \frac{D^{p'-1}(s)}{y_{c^*}^{1/p}(s)}\,\ud s \sim \int_0^u \frac{s^{\alpha(p'-1)}}{s^{p'-1}}\,\ud s 
	= \int_0^u s^\frac{\alpha-1}{p-1}\, \ud s \quad \mbox{ as }\,\, u\to 0+,
\end{equation}
we conclude that the following two cases occur:
\begin{itemize}
	\item [(a)]
	$\xi_2=+\infty$ if and only if $p+\alpha \leq 2$;
	\item [(b)]
	$\xi_2<+\infty$ if and only if $p+\alpha>2$,
\end{itemize}
see Figure \ref{f:as0} for geometric interpretation.

\begin{figure}[h!]
	\centering
	\includegraphics[width=8cm]{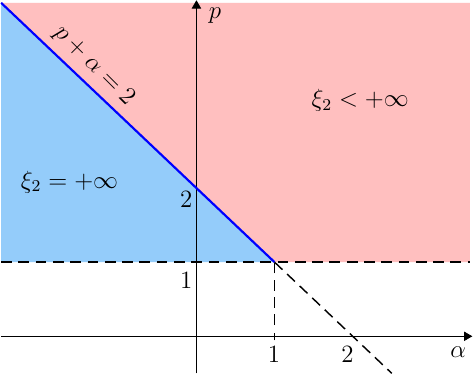}
	\caption{Visualization of cases (a) and (b), leading to $\xi_2$ finite or infinite}
	\label{f:as0}
\end{figure}

Observe that for any $\alpha>1$, the profile $u=u(\xi)$ is always right compactly supported, i.e., $u\equiv 0$ in $[\xi_2,+\infty)$, $\xi_2 \in \R$.
If $\alpha=0$ and $1<p\leq2$, the profile does not attain 0 for any finite $\xi$. This result is consistent with that from \cite{MM} for $p=2$ and $D \in C^1[0,1]$  strictly positive in $[0,1]$.

In case (b), we can also study the one-sided derivative $u'(\xi_2-)$. In particular, differentiating \eqref{xi} yields
\begin{equation*}
	\frac{\ud \xi}{\ud u}
	=
	-\frac{\ud}{\ud u} \int_{\frac{1}{2}}^u \frac{D^{p'-1}(s)}{y_{c^*}^{1/p}(s)}\,\ud s
	=
	- \frac{D^{p'-1}(u)}{y_{c^*}^{1/p}(u)}, \quad u \in (0,1).
\end{equation*}
Since $D(u)\sim u^\alpha$, $y(u) \sim u^{p'}$ as $u\to 0+$, we have
\begin{equation*}
	\left.\frac{\ud \xi}{\ud u}\right|_{u=0+} \sim - u^\frac{\alpha-1}{p-1} \to 
	\begin{cases}
		0\quad &\mbox{if}\quad \alpha > 1\\
		\mathrm{const.}<0 \quad &\mbox{if}\quad \alpha=1 \\
		-\infty \quad &\mbox{if}\quad  \alpha<1
	\end{cases}
\quad \mbox{ as } u \to 0+.
\end{equation*}
From an inverse perspective, we obtain the following classification for the profile $u=u(\xi)$:
\begin{equation*}
	u'(\xi_2-) = 
	\begin{cases}
		-\infty\quad &\mbox{if}\quad \alpha > 1,\\
		\mathrm{const.}<0 \quad &\mbox{if}\quad \alpha=1, \\
		0 \quad &\mbox{if}\quad  \alpha<1.
	\end{cases}
\end{equation*}
Therefore, if $p+\alpha>2$ and $\alpha<1$ we have $u'(\xi_2-)=u'(\xi_2+)=0$.

\subsection{Asymptotics near 1}

Let us assume that $D(u)\sim (1-u)^\beta$ and $g(u) \sim (1-u)^\gamma$ as $u\to 1-$ for some $\beta \in \R$, $\gamma>0$.
Since the equation \eqref{ODE-y} is not separable on $(\theta,1)$, the asymptotic analysis becomes more involved than in the previous case.
However, we can apply the same reasoning as in \cite[Section 5.1]{DrZa-tw}, where we investigated asymptotic properties of solutions in the absence of convection.
In fact, this technique yields the same results also when $h(u) \geq 0$ instead of $h\equiv 0$.
Replacing $c$ by $c+h(u)$ in \cite[Section 5.1]{DrZa-tw}, we obtain
the same classification of solutions as in Theorems 5.1 and 5.2 therein. In our current notation, these theorems read as follows.

\begin{theorem}
\label{t:as1}
Let $D(u) \sim (1-u)^\beta$, $g(u) \sim (1-u)^\gamma$ as $u\to 1-$ where $\beta \in \R$ and $\gamma>0$ are such that
\begin{equation*}
	-1<\gamma+\frac{\beta}{p-1}\leq \frac{1}{p-1}
\end{equation*}
for given $p>1$.
If
\begin{equation*}
	\frac{\gamma-\beta+1}{p}<1,
\end{equation*}
then $\xi_1>-\infty$. If
\begin{equation*}
	\frac{\gamma-\beta+1}{p}\geq1,
\end{equation*}
then $\xi_1=-\infty$.
\end{theorem}

\begin{theorem}
\label{t:as1-2}
Let $D(u) \sim (1-u)^\beta$, $g(u) \sim (1-u)^\gamma$ as $u\to 1-$ where $\beta \in \R$ and $\gamma>0$ are such that
\begin{equation*}
	\gamma+\frac{\beta}{p-1} > \frac{1}{p-1}
\end{equation*}
for given $p>1$. If $\gamma<1$ then $\xi_1>-\infty$. If $\gamma \geq 1$ then $\xi_1=-\infty$.
\end{theorem} 

\begin{remark}
To visualize conditions from Theorems \ref{t:as1}, \ref{t:as1-2},
we introduce the following sets:
\[
\begin{split}
	\mathcal{M}_1^1&:=\{(\gamma,\beta)\in\mathbb{R}^2: \gamma>0, -1<\gamma+\frac{\beta}{p-1}\leq \frac{1}{p-1}, \gamma-\beta+1<p\},\\
	\mathcal{M}_1^2&:=\{(\gamma,\beta)\in\mathbb{R}^2: \gamma>0, -1<\gamma+\frac{\beta}{p-1}\leq \frac{1}{p-1}, \gamma-\beta+1\geq p\},\\
	\mathcal{M}_1^3&:=\{(\gamma,\beta)\in\mathbb{R}^2: \gamma>0, \gamma+\frac{\beta}{p-1}>\frac{1}{p-1}, \gamma<1\},\\
	\mathcal{M}_1^4&:=\{(\gamma,\beta)\in\mathbb{R}^2: \gamma>0, \gamma+\frac{\beta}{p-1}>\frac{1}{p-1}, \gamma\geq 1\}.
\end{split}
\]
Then $\xi_1>-\infty$ if and only if $(\gamma,\beta)\in \mathcal{M}_1^1 \cup \mathcal{M}_1^3$ and
$\xi_1=-\infty$ if and only if $(\gamma,\beta)\in \mathcal{M}_1^2 \cup \mathcal{M}_1^4$.
See Figure \ref{f:as1} for the case $p=2$.

\begin{figure}[h!]
	\centering
	\includegraphics[width=8cm]{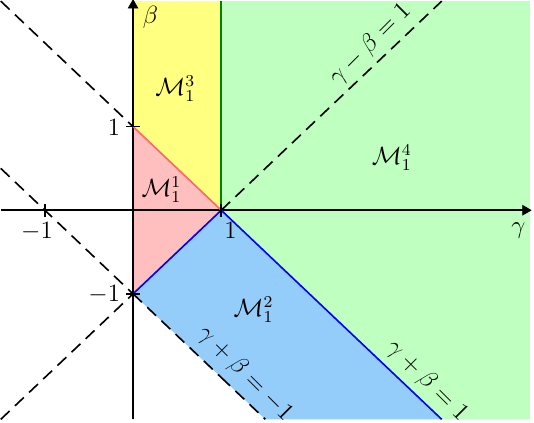}
	\caption{Visualization of the sets $\mathcal{M}_1^1$, $\mathcal{M}_1^2$, $\mathcal{M}_1^3$ and $\mathcal{M}_1^4$ for $p=2$}
	\label{f:as1}
\end{figure}

Moreover, one can show exactly as in \cite[Remark 6.4]{DZ22} that if $\xi_1 \in \R$, i.e., $(\gamma,\beta) \in \mathcal{M}_1^1 \cup \mathcal{M}_1^3$, then $u'(\xi_1)=0$.

\end{remark}

\section*{Acknowledgement}
Pavel Dr\'{a}bek was supported by the Grant Agency of the Czech Republic (GA\v{C}R) under Grant No.~22-18261S.

\appendix

\section{Auxiliary lemmas for Proposition \ref{p:prop-u}}
\label{ap:A}

\begin{lemma}
	\label{l:A1}
	Let $u\in C^1(\R)$ be a solution of the initial value problem 
	\begin{equation}
		\label{2IVP}
		\begin{cases}
			\left(D(u)|u'|^{p-2}u'\right)' =-\left(c+h(u)\right)u',\\
			u(\xi_0)=u_0 \in (0,1),\,\, u'(\xi_0)=0.
		\end{cases}
	\end{equation}
	Then $u$ does not verify part (c) of Definition \ref{def:sol}.
\end{lemma}

\begin{proof}
	Integrating the equation in \eqref{2IVP} and using the initial conditions yields
	\begin{equation} \label{ivp_int}
		D(u(\xi))|u'(\xi)|^{p-2}u'(\xi)=c(u_0-u(\xi))+H(u_0)-H(u(\xi)), \quad \xi \in \R.
	\end{equation}
	Put
	\begin{equation*}
		\mathcal{S}_p(\nu) \coloneqq |\nu|^{p-2}\nu \, \mbox{ for }\, \nu \ne 0, \,\,
		\mathcal{S}_p(0)=0, \quad p>1.
	\end{equation*}
	Since $\mathcal{S}_{p'}$ is the inverse function to $\mathcal{S}_p$, equation \eqref{ivp_int} is for $D(u(\xi)) \ne 0$ equivalent to
	\begin{equation}
		\label{u'}
		u'(\xi)=\mathcal{S}_{p'}\left(\frac{1}{D(u(\xi))}\left[c(u_0-u(\xi))+H(u_0)-H(u(\xi))\right]\right).
	\end{equation}
	
	If $1<p\leq 2$ then, due to $D\in C^1(0,1)$, $H \in C^1[0,1]$ and $p'>2$, the right-hand side of \eqref{u'} is locally Lipschitz continuous in $u$. Hence $u(\xi)=u_0$, $\xi \in \R$, is a unique solution of \eqref{2IVP} in $\R$, and therefore does not verify part (c) of Definition \ref{def:sol}.
	
	If $p>2$, i.e., $1<p'< 2$, then the right-hand side of \eqref{u'} is not Lipschitz continuous only at one point $u=u_0$, but it is one-sided Lipschitz continuous there
	due to the fact that $\mathcal{S}_{p'}(\nu)=|\nu|^{p'-2}\nu$ satisfies one-sided Lipschitz condition. 
	Therefore, either $u(\xi)=u_0$, $\xi \in (-\infty,\xi_0]$ is a unique solution of \eqref{2IVP} in $(-\infty,\xi_0]$, or $u(\xi)=u_0$, $\xi \in [\xi_0,+\infty)$, is a unique solution of \eqref{2IVP} in $[\xi_0,+\infty)$. In either case, part (c) of Definition \ref{def:sol} is not satisfied. 
\end{proof}

\begin{lemma}
	\label{l:A2}
	Let $u$ be a solution of \eqref{2BVP-u} and let $\xi_0 \in \R$ be such that $u(\xi_0) \in (0,1)$, $u'(\xi_0)=0$ and
	\[
	\left.\left(D(u(\xi))|u'(\xi)|^{p-2}u'(\xi))\right)'\right|_{\xi=\xi_0}<0.
	\]
	Then $u$ has a strict local maximum at $\xi_0$.
\end{lemma}

\begin{proof}
	Let us recall that it follows from Remark \ref{r:diff-v} that both functions $\xi \mapsto D(u(\xi))|u'(\xi)|^{p-2}u'(\xi)$ and $\xi \mapsto |u'(\xi)|^{p-2}u'(\xi)$ are continuously differentiable in a small neighborhood of $\xi$.
	We have
	\begin{equation*}
		\begin{split}
			0>\left.\left(D(u(\xi))|u'(\xi)|^{p-2}u'(\xi))\right)'\right|_{\xi=\xi_0}
			= {}&
			\left.\frac{\ud D}{\ud u}\right|_{u=u(\xi_0)} \underbrace{|u'(\xi_0)|^{p}}_{=0} 
			+
			D(u(\xi_0))\left.\left(|u'(\xi)|^{p-2}u'(\xi)\right)'\right|_{\xi=\xi_0}.
		\end{split}
	\end{equation*}
	Since $D(u(\xi_0))>0$, we get $\left.\left(|u'(\xi)|^{p-2}u'(\xi)\right)'\right|_{\xi=\xi_0}<0$, and
	therefore, $|u'(\xi)|^{p-2}u'(\xi)$ is strictly decreasing in $\xi_0$ and equal to 0 at $\xi=\xi_0$.
	Since the power $\mathcal{S}_p(\nu)=|\nu|^{p-2}\nu$ is strictly increasing, $u'(\xi)$ is strictly decreasing at $\xi=\xi_0$.
	Hence $\xi_0$ is the point of strict local maximum of $u$.
\end{proof}

\section{Technical lemmas for Section \ref{s:Ex}}
\label{ap:B}

\begin{lemma}[Technical inequalities]
	\label{l:tech}
	Let $a>0$, $b>0$. Then
	\begin{itemize}
		\item [\emph{(i)}]
		for $r \geq 2$ we have
		\begin{equation*}
			a^r+ra^{r-1}b + b^r \leq (a+b)^r\,;
		\end{equation*}
		
		\item [\emph{(ii)}]
		for $1<r<2$ we have
		\begin{equation*}
			a^r+ra^{r-1}b + b^r \leq \hat{k}(r)(a+b)^r,
		\end{equation*}
		where
		\begin{equation*}
			\hat{k}(r)=\dfrac{1+r(r-1)^{\frac{1}{r-2}}+(r-1)^{\frac{r}{r-2}}}{\vphantom{A^{A^{A^{A^a}}}}\left(1+(r-1)^{\frac{1}{r-2}}\right)^r}\,.
		\end{equation*}
	\end{itemize}
\end{lemma}

\begin{proof}
	We put	$t=\frac{b}{a}>0$ and write the inequality in an equivalent form
	\begin{equation*}
		f(t) \coloneqq \frac{1+rt+t^r}{(1+t)^r} \leq \hat{k}(r), 
	\end{equation*}
	where we set $\hat{k}(r)=1$ for $r \geq 2$. Then the optimal choice for $\hat{k}(r)$ would be $\hat{k}(r)= \max_{t\geq 0} f(t)$, if this maximum exists. Indeed, it does.  Namely, $f$ is a continuously differentiable function on $[0,+\infty)$ satisfying $f(0)=1=\lim\limits_{t \to +\infty} f(t)$. An elementary calculation yields that $t_1=(r-1)^{\frac{1}{r-2}}$ is the only stationary point of $f$ in $(0,+\infty)$.
	
	Part (i). It is clear that equality holds for $r=2$. Let $r>2$. Then $f(1)=\frac{2+r}{2^r}<1$. Hence $t_1=(r-1)^{\frac{1}{r-2}}$ is the point of global minimum of $f$, $0<f(t_1)\leq f(1)<1$ and therefore $\max_{t\geq 0} f(t)=f(0)=1$.
	
	Part (ii). Let $1<r<2$. Then $f(1)=\frac{2+r}{2^r}>1$ and hence $t_1$ is the point of global maximum of $f$ in $[0,+\infty)$ with
	\begin{equation*}
		\hat{k}(r)=f(t_1)= \frac{1+r(r-1)^{\frac{1}{r-2}}+(r-1)^\frac{r}{r-2}}{\vphantom{A^{A^{A^{A^a}}}}\left(1+(r-1)^\frac{1}{r-2}\right)^r}\,.
	\end{equation*}
\end{proof}

\begin{lemma}[Inequality \eqref{yH-theta}]
\label{l:y<H}
Assume that $h(u) \geq 0$ in $[0,1]$ and let $y_0=y_0(u)$ be a solution of the initial value problem \eqref{bivp-comb} with $c=0$. 
If
\begin{equation}
\label{eq_lemma}
H^{p'}(1) \leq k(p) \int_0^1 f(u)\,\ud u,
\end{equation} 
where $k=k(p)$ is given by \eqref{k-def},
then
\begin{equation*}
	y_0^\frac{1}{p'}(\theta) > H(\theta).
\end{equation*}
\end{lemma}

\begin{proof}
We proceed by contradiction, that is, we assume that
\begin{equation*}
	y_0^\frac{1}{p'}(\theta) \leq H(\theta).
\end{equation*}
Since $f>0$ on $(\theta,1)$, it follows from the equation in \eqref{bivp-comb} that $y_0(u)>0$ for all $u \in (\theta,1)$. 
Set $z(u) \coloneqq y_0^\frac{1}{p'}(u)$. Then $z(u)>0$ in $(\theta,1)$, $z(1)=0$,
\begin{equation}
	\label{zH-theta}
	z(\theta) \leq H(\theta)
\end{equation}
and $z=z(u)$ satisfies the equation
\begin{equation}
\label{z-ODE1}
	[z^{p'}(u)]'=p'h(u)z^{p'-1}(u)-p'f(u), \quad u \in (\theta,1),
\end{equation}
or, equivalently,
\begin{equation}
\label{z-ODE2}
	z'(u)=h(u)-\frac{f(u)}{z^{p'-1}(u)}, \quad u \in (\theta,1).
\end{equation}
Integrating \eqref{z-ODE1} and using the mean value theorem, we obtain
\begin{equation}
	\label{zp'-theta}
	\begin{split}
		z^{p'}(\theta) &{}=
		z^{p'}(1) - p'\int_{\theta}^1 h(\sigma)z^{p'-1}(\sigma)\,\ud\sigma + p'\int_{\theta}^1 f(\sigma)\,\ud\sigma \\
		&{}=
		-p'z^{p'-1}(\tau_0) (H(1)-H(\theta)) + p'\int_{0}^1 f(\sigma)\,\ud\sigma
	\end{split}
\end{equation}
for some $\tau_0 \in (\theta,1)$. 
From \eqref{z-ODE2} we have
\begin{equation*}
	z(\tau_0)-z(\theta)=\int_{\theta}^{\tau_0} h(\sigma)\,\ud\sigma - \int_{\theta}^{\tau_0} \frac{f(\sigma)}{z^{p'-1}(\sigma)}\,\ud\sigma
	<
	H(\tau_0)-H(\theta) 		
\end{equation*}
and hence
\begin{equation}
	\label{ztau0-ztheta-p}
	z(\tau_0)<z(\theta)+ H(1)-H(\theta)
\end{equation}
thanks to the monotonicity of $H$ (in particular, $h\geq 0$ implies that $H$ is nondecreasing).
It follows from \eqref{zp'-theta}, \eqref{ztau0-ztheta-p} together with \eqref{zH-theta} 
\begin{equation}
\label{Hp'theta}
H^{p'}(\theta) > -p'\left(H(\theta)+[H(1)-H(\theta)]\right)^{p'-1}(H(1)-H(\theta))+p'\int_0^1 f(\sigma)\,\ud \sigma.
\end{equation}
Next we proceed separately for $p=2$, $1<p<2$ and $p>2$.

\smallskip
\noindent
\emph{Case 1:} $p=2$.
Since $p'=2$, \eqref{Hp'theta} becomes
\begin{equation*}
	H^2(\theta)>-2\left(H(\theta)+[H(1)-H(\theta)]\right)(H(1)-H(\theta))+2\int_0^1 f(\sigma)\,\ud\sigma.
\end{equation*}
Reorganizing the terms in the above inequality and using \eqref{eq_lemma}, we obtain
\begin{equation*}
	\begin{split}
		0 &{}> -H^2(\theta)-2H(\theta)(H(1)-H(\theta))-\left(H(1)-H(\theta)\right)^2 - \left(H(1)-H(\theta)\right)^2 
		+2\int_0^1 f(\sigma)\,\ud\sigma \\
		&= -H^2(1)-\left(H(1)-H(\theta)\right)^2+2\int_0^1 f(\sigma)\,\ud\sigma
		>2\left(\int_0^1 f(\sigma)\,\ud\sigma - H^2(1)\right) \geq 0,
	\end{split}
\end{equation*}
a contradiction.

\smallskip
\noindent
\emph{Case 2:} $1<p<2$.
Since $p'>2$, we use the well-known inequality
\begin{equation*}
	(a+b)^r \leq 2^{r-1}(a^r+b^r), \quad a,b>0, \,\, r>1,
\end{equation*}
with $a=H(\theta)$, $b=H(1)-H(\theta)$, $r=p'-1$ in \eqref{Hp'theta} and obtain
\begin{equation*}
	\begin{split}
		H^{p'}(\theta) &{}>-p'\left(H(\theta)+[H(1)-H(\theta)]\right)^{p'-1}(H(1)-H(\theta))+p'\int_0^1 f(\sigma)\,\ud\sigma \\
		& \geq -p'2^{p'-2}\left(H^{p'-1}(\theta)+\left[H(1)-H(\theta)\right]^{p'-1}\right)(H(1)-H(\theta))+p'\int_0^1 f(\sigma)\,\ud\sigma.
	\end{split}
\end{equation*} 
Hence
\begin{equation*}
\begin{split}
0 > &{} -H^{p'}(\theta)-p'2^{p'-2}H^{p'-1}(\theta)(H(1)-H(\theta))-p'2^{p'-2}\left(H(1)-H(\theta)\right)^{p'}+p'\int_0^1 f(\sigma)\,\ud\sigma \\
& = - H^{p'}(\theta) - p'H^{p'-1}(\theta)(H(1)-H(\theta)) - (H(1)-H(\theta))^{p'} \\
& \quad\, + (1-p'2^{p'-2})(H(1)-H(\theta))^{p'} + (p'-p'2^{p'-2})H^{p'-1}(\theta)(H(1)-H(\theta)) + p'\int_0^1 f(\sigma)\,\ud\sigma
\end{split}
\end{equation*}
and, using the inequality from Lemma \ref{l:tech}\,(i) with $a=H(\theta)$, $b=H(1)-H(\theta)$ and $r=p'$,
\begin{equation*}
	\begin{split}
		0 &{}> 
		- (H(\theta)+(H(1)-H(\theta)))^{p'} + (1-p'2^{p'-2})(H(1)-H(\theta))^{p'} \\
		& \quad\quad  + (p'-p'2^{p'-2})H^{p'-1}(\theta)(H(1)-H(\theta)) + p'\int_0^1 f(\sigma)\,\ud\sigma.
	\end{split}
\end{equation*}
Then $0 \leq H(\theta)\leq H(1)$ implies
\begin{equation*}
	0>-H^{p'}(1)+(1-p'2^{p'-2})H^{p'}(1)+(p'-p'2^{p'-2})H^{p'}(1)+p'\int_0^1 f(\sigma)\,\ud\sigma
\end{equation*}
and from \eqref{eq_lemma} we conclude
\begin{equation*}
	0>-p'(2^{p'-1}-1)H^{p'}(1)+p'\int_0^1 f(\sigma)\,\ud\sigma \geq 0,
\end{equation*}
a contradiction.

\smallskip
\noindent
\emph{Case 3:} $p>2$.
Since $1<p'<2$, we now use the well-known inequality
\begin{equation*}
	(a+b)^r \leq a^r+b^r,\quad a,b>0, \,\, 0<r<1,
\end{equation*}
with $a=H(\theta)$, $b=H(1)-H(\theta)$, $r=p'-1$ in \eqref{Hp'theta} and obtain
\begin{equation*}
	0>-H^{p'}(\theta)-p'\left(H^{p'-1}(\theta)+\left[H(1)-H(\theta)\right]^{p'-1}\right)(H(1)-H(\theta))+p'\int_0^1 f(\sigma)\,\ud\sigma,
\end{equation*}
i.e.,
\begin{equation*}
	0>-H^{p'}(\theta)-p'H^{p'-1}(\theta)(H(1)-H(\theta))-p'\left(H(1)-H(\theta)\right)^{p'}+p'\int_0^1 f(\sigma)\,\ud\sigma,
\end{equation*}
or equivalently
\begin{equation*}
	\begin{split}
		0>-H^{p'}(\theta)-p'H^{p'-1}(\theta)(H(1)-H(\theta))-\left(H(1)-H(\theta)\right)^{p'} \\
		-(p'-1)\left(H(1)-H(\theta)\right)^{p'}+p'\int_0^1 f(\sigma)\,\ud\sigma.
	\end{split}
\end{equation*}
For $a>0$, $b>0$, $r\in(1,2)$ we have $a^r+ra^{r-1}b+b^r \leq \hat{k}(r)(a+b)^r$
by the technical Lemma \ref{l:tech}\,(ii). We apply it with $a=H(\theta)$, $b=H(1)-H(\theta)$, $r=p'$:
\begin{equation*}
	0 >-\hat{k}(p')\left(H(\theta)+(H(1)-H(\theta))\right)^{p'}-(p'-1)\left(H(1)-H(\theta)\right)^{p'}+p'\int_0^1 f(\sigma)\,\ud\sigma.
\end{equation*}
But \eqref{eq_lemma} yields
\begin{equation*}
	0  > -(\hat{k}(p')+(p'-1))H^{p'}(1)+p'\int_0^1 f(\sigma)\,\ud\sigma \geq 0,
\end{equation*}
a contradiction.
\end{proof}

\end{document}